\documentclass[12pt]{article}
\usepackage{amsmath, amssymb, amsthm}
\usepackage{graphicx}
\usepackage{enumerate}
\date{}

\begin{document}

\title{\bf Analysis of Caputo impulsive fractional order differential equations with applications}
\author{Lakshman Mahto$^{1}$, Syed Abbas$^{1}$, Angelo Favini$^{2}$ \thanks{
E-mails: favini@dm.unibo.it \, sabbas.iitk@gmail.com}
\\\\
$^{1}$School of Basic Sciences \\
Indian Institute of Technology Mandi \\
Mandi, H.P., 175001, India. \\\\
$^{2}$Dipartimento di Matematica \\
Universita di Bologna \\
Pizza di Porta S. Donato, 5, \\
40126 - Bologna, Italy
} \vskip .1in

\begin{titlepage}

\maketitle

\noindent
\begin{abstract}

We use Sadavoskii's fixed point method to investigate the existence and uniqueness of
solutions of Caputo impulsive fractional differential equations of order $\alpha
\in (0,1)$ with one example of impulsive logistic model and few other examples as well. We also discuss Caputo impulsive fractional differential equations with finite delay. The results proven are new and compliment the existing one.

\vspace{1em}

\end{abstract}
\end{titlepage}

\maketitle \numberwithin{equation}{section}
\newtheorem{theorem}{Theorem}[section]
\newtheorem{lemma}[theorem]{Lemma}
\newtheorem{example}[theorem]{Example}
\newtheorem{proposition}[theorem]{Proposition}
\newtheorem{corollary}[theorem]{Corollary}
\newtheorem{remark}[theorem]{Remark}
\newtheorem{definition}[theorem]{Definition}

\setcounter{page}{2}

\section{Introduction}
Dynamics of many evolutionary processes from various field as population dynamics, control theory, physics, biology, medicine etc. undergo abrupt changes at certain moment of times like earthquake, harvesting, shock etc. These perturbations can be well-approximated as instantaneous change of state or impulses. These processes are modelled by {\em impulsive differential equations}. In 1960, Milman and Myshkis introduced impulsive differential equations in their paper \cite{mil}. Based on their work, several monographs have been published by many authors like Semoilenko and Perestyuk \cite{sam}, Lakshmikantham et. al. \cite{laksh}, Bainov and Semoinov \cite{bain, bain1}, Bainov and Covachev \cite{bain2} and Benchohra et. al. \cite{ben}.  All the authors mentioned above have considered impulsive differential equations as ordinary differential equations coupled with impulsive effects. They considered the impulsive effects as difference equations being satisfied at impulse times. So, the solutions are piecewise continuous with discontinuities at impulse times. In the field like biology, population dynamics, etc., problems with hereditary are best modelled by delay differential equations \cite{hale}. Problems associated with impulsive effects and hereditary property are modelled by impulsive delay differential equations.

The origin of fractional calculus (derivatives $\frac{d^{\alpha}}{dt^{\alpha}}f$ and integrals $I^\alpha f$ of
arbitrary order $\alpha>0$) goes back to Newton and Leibniz in the 17th
century. In a letter correspondence with Leibniz, l'Hospital asked "What if the order
of the derivative is $\frac{1}{2}"?$ Leibniz replied, "Thus it follows that
will be equal to $x\sqrt{dx:x},$ an apparent paradox, from which one day
useful consequences will be drawn." This letter of Leibniz was in September, 1695. So 1695 is considered as the birthday of fractional calculus.
Fractional order differential equations are
generalizations of classical integer order  differential equations
and are increasingly used to model problems in fluid dynamics, finance
and other areas of application. Recent investigations have shown
that often times physical systems can be modelled more accurately
using fractional derivative formulations \cite{mainardi}. There
are several excellent monographs available on this
field \cite{diet, kilbas, miller,pod1,saba, samko}. In
\cite{kilbas}, the authors give a recent and up-to-date description of the
developments of fractional differential and fractional
integro-differential equations including applications. The
existence and uniqueness of solutions to fractional differential equations has been considered by many authors  \cite{abbas, abbas1, ahmad,
hadid1, ibrahim, laksh1}. Impulsive fractional differential equations represent a real framework for mathematical modelling to real world problems. Significant progress has been made in the theory of impulsive fractional differential equations \cite{ben, ben1, ben2, fecken}. Xu et. al. in his paper \cite{xu} has described an impulsive delay fishing model. \\
Fractional derivatives arise naturally in mathematical problems, for $\alpha>0$ and a function $f: [0,T] \rightarrow \mathbb{R}$ recall \cite[Definitions 3.1, 2.2]{diet} the
\begin{itemize}
\item[(a)] \emph{Caputo fractional derivative}
\[
  ^{C}D^{\alpha} f(t)=\frac{1}{\Gamma (1-\alpha)}
  \int_0^t(t-s)^{-{\alpha}}f'(s)ds,
\]

\item[(b)] \emph{Riemann-Liouville fractional derivative}
\[
  D^{\alpha} f(t)=\frac{d^{\alpha}}{dt^{\alpha}}f(t)=\frac{1}{\Gamma (1-\alpha)}
  \frac{d}{dt}\int_0^t(t-s)^{-{\alpha}}f(s)ds,
\]
\end{itemize}
provided the right hand sides exist point-wise on $[0,T]$ ($\Gamma$ denotes the gamma function). Using the  \emph{Riemann-Liouville fractional integral} \cite[Definition 2.1]{diet} $I^{\alpha}_0 f(t)=\frac{1}{\Gamma (\alpha)}\int_0^t(t-s)^{\alpha-1}f(s)ds$, we have $^{C}D^{\alpha} f(t) =I^{1-\alpha}_0\frac{d}{dt}f(t)$ and
$D^{\alpha} f(t)=\frac{d}{dt}I^{1-\alpha}_0f(t)$. $I^{\alpha}_0f$ exists, for instance, for all $\alpha >0,$ if $f\in C^0([0,T])\cap L^1_{loc}([0,T])$, moreover $I^{\alpha}_0f(0) = 0$.

One can see that the both the fractional derivatives are actually non-local operator because integral is a non-local operator.
Moreover, calculating time fractional derivative of a function at some
time requires all the past history and hence fractional derivatives can be used for modelling systems with
memory. Fractional differential equations can be formulated using both Caputo or Riemann-Liouville fractional derivatives.
A \emph{Riemann-Liouville initial value problem} can be stated as follows
\begin{align}
\begin{split}
D^{\alpha}x(t) &=
f(t,x(t)), \; t \in [0,T],
\\ D^{\alpha-1}x(0) &=x_0,
\end{split}
\label{EE_1}
\end{align}
or equivalently,
$x(t)=\frac{x_0t^{\alpha-1}}{\Gamma(\alpha)}+\int_0^t(t-s)^{\alpha-1}f(s,x(s))ds$
in its integral representation \cite[Theorem 3.24]{kilbas}. For a physical interpretation of the initial conditions in \eqref{EE_1} see \cite{gian,pod,pod2}. If derivatives of Caputo type are used instead of Riemann-Liouville type then initial conditions for the corresponding Caputo fractional differential equations can be formulated as for classical ordinary differential equations, namely $x(0)=x_0$.

Our main objective is to discuss existence and uniqueness of
solutions of the following \emph{ impulsive fractional differential equation of Caputo type} in a Banach space $X$ with norm $|.|,$
\begin{align}
\begin{split}
^{C}D^{\alpha}x(t) & =
f(t,x(t)), \; t \in I=[0,T], \; t\neq t_k,
\\ \Delta x(t)|_{t=t_k} & = I_k(x(t_k^-)), k =1,2,...,m , \; \\   x(0) &=x_0.
\end{split}
\label{EE_2}
\end{align}
where $f\in C(I\times X,X), I_k:X\rightarrow X$ and $x_0\in X$.                 $0=t_0<t_1<t_2<\cdots<t_m<t_{m+1}=T, \Delta x(t)|_{t=t_k}=x(t_k^+)-x(t_K^-), x(t_k^+)=\lim_{h\rightarrow 0}x(t+h)\ \mbox{and} \ x(t_k^-)=\lim_{h\rightarrow 0}x(t-h).$ We break our function $f$ into two components which satisfy different conditions. We observed that these kind of functions occurs in ecological modelling. We have given the example of logistic equation in the last section. Our main tool is Sadovskii's fixed point theorem.

Now, we define some important spaces and norm which will encounter frequently:
\begin{align*}
PC(I,X) = \{ x:[0,T] \rightarrow X| x\in C((t_k,t_{k+1}],X), k = 0,1,2 \cdots m,  \\   x(t_k^+) \ \mbox{and} \ x(t_k^-) \ \mbox{exist}, \ x(t_k)=x(t_k^-) \}.
\end{align*} with sup-norm $\|.\|$, defined by $\|x\|=\sup \{|x(t)|: t\in I\}.$ \\ $\mathcal{C}=C([-r,0],X)$  with sup-norm $\|x\|_r=\sup \{|x(t)|: t\in [-r,0]\}.$

\begin{definition}
A solution of fractional differential equation \eqref{EE_2} is a piecewise continuous function $x\in PC([0,T],X)$ and which satisfied \eqref{EE_2}.
\end{definition}

\begin{definition}(Definition 11.1 \cite{zeid})
Kuratowskii non-compactness measure: Let $M$ be a bounded set in metric space $(X,d)$, then Kuratowskii non-compactness measure, $\mu(M)$ is defined as $\inf \{\epsilon: \ \mbox{M covered by a finite many} \\  \mbox{sets such that the diameter of each set} \ \leq \epsilon\}.$
\end{definition}

\begin{definition}(Definition 11.6 \cite{zeid})
Condensing map: Let $\Phi:X\rightarrow X$ be a bounded and continuous  operator on Banach space $X$ such that $\mu(\Phi(B))<\mu(B)$ for all bounded set $B\subset D(\Phi)$, where $\mu$ is the Kuratowskii non-compactness measure, then $\Phi$ is called condensing map.
\end{definition}

\begin{theorem} (\cite{sado})
Let B be a convex, bounded and closed subset of a Banach space $X$
and $\Phi: B \rightarrow B$ be a condensing map. Then, $\Phi$
 has a fixed point in $B.$ \label{thm1}
\end{theorem}

\begin{lemma}(Example 11.7, \cite{zeid})
A map $ \Phi=\Phi_1+\Phi_2:X\rightarrow X$ is $k-$ contraction with $0\leq k< 1$ if
\begin{enumerate}
\item[(a)] $\Phi_1$ is $k-$ contraction i.e. $|\Phi_1(x)-\Phi_1(y)| \leq k|x-y|$ and
\item[(b)] $\Phi_2$ is compact,
\end{enumerate}
and hence $\Phi$ is a condensing map. \label{lem1*}
\end{lemma}

The structure of the paper is as follows. In section 2, we prove existence (Theorem \ref{thm2}) and uniqueness
of solutions to \eqref{EE_2}. We show in section 3, the existence and uniqueness of solutions for a general class of impulsive functional differential equations of fractional order $\alpha \in (0,1)$. In the last section 4, we give some examples in favor of our sufficient conditions.

\section{Impulsive fractional differential equation}

Consider the initial value problem \eqref{EE_2} on the cylinder $R=\{(t,x) \in \mathbb{R} \times X : t\in[0,T],x\in B(0,r)\}$ for some fixed $T>0,r>0$, and assume that there exist $p \in (0,\alpha), \alpha \in (0,1), \ M_{1}, M_{2}, L_1 \in L_{1/p}([0,T],\mathbb{R}^{+})$ and functions $f_1, f_2 \in C(R,X)$ such that $f = f_1 + f_2$ and the following assumptions are satisfied
\begin{enumerate}
\item[(A.1)] $f_1$ is bounded and Lipschitz, in particular, $|f_{1}(t,x)| \le M_{1}(t)$ and $|f_{1}(t,x)-f_{1}(t,y)| \le L_1(t)\ |x-y|$ for all
$(t,x), (t,y) \in R$,

\item[(A.2)] $f_{2}$ is compact and bounded, in particular, $|f_{2}(t,x)| \le M_{2}(t)$ for all $(t,x) \in R$,

\item[(A.3)] $I_k\in C(X,X)$ such that $|I_k(x)|\leq l_1$ and $|I_k(x)-I_k(y)|\leq l_2|x-y|.$
\end{enumerate}

\begin{lemma}(Fecken et al, Lemma 2. \cite{fecken}) The initial value problem \eqref{EE_2} is equivalent to the non-linear integral equation
\begin{eqnarray}
\ x(t) &=& x_0+ \frac{1}{\Gamma(\alpha)}\int_0^t(t-s)^{\alpha-1}f_1(s,x(s))ds \nonumber
\\ && + \frac{1}{\Gamma(\alpha)}\int_0^t(t-s)^{\alpha-1}f_2(s,x(s))ds, \; t\in [0,t_1] \nonumber
\\
&=& x_0+I_1(x(t_1^-))+\frac{1}{\Gamma(\alpha)}\int_0^t(t-s)^{\alpha-1}f_1(s,x(s))ds \nonumber \\ && +\frac{1}{\Gamma(\alpha)}\int_0^t(t-s)^{\alpha-1}f_2(s,x(s))ds, \; t\in (t_1,t_2] \nonumber \\
 &=& x_0+\sum_{k=1}^2I_k(x(t_k^-))+\frac{1}{\Gamma(\alpha)}\int_0^t(t-s)^{\alpha-1}f_1(s,x(s))ds \nonumber
\\ && +\frac{1}{\Gamma(\alpha)}\int_0^t(t-s)^{\alpha-1}f_2(s,x(s))ds,  \; t\in (t_2,t_3] \nonumber
\\
 &=& x_0+\sum_{k=1}^mI_k(x(t_k^-))+\frac{1}{\Gamma(\alpha)}\int_0^t(t-s)^{\alpha-1}f_1(s,x(s))ds \nonumber
\\ && + \frac{1}{\Gamma(\alpha)}\int_0^t(t-s)^{\alpha-1}f_2(s,x(s))ds,  \;  t\in (t_m,T]. \label{EE_3}
\end{eqnarray}
In other words, every solution of the integral equation \eqref{EE_3} is
also solution of our original initial value problem \eqref{EE_2} and conversely.
\end{lemma}

\begin{theorem}[Existence of solutions]
Under the assumptions $(A.1)-(A.3)$ the problem \eqref{EE_2} has
at least one solution in $[0,T]$, provided that
$$\gamma_1=ml_2+\frac{c \|L_1\|{\frac{1}{p}} T^{\alpha-p}}{\Gamma(\alpha+1)}<1, \quad \mbox{where} \  c=\Big(\frac{1-p}{\alpha-p}\Big)^{1-p}.$$ \label{thm2}
\end{theorem}
\begin{proof}
Let $B_\lambda$ be the closed bounded and convex subset of $PC([0,T],X),$ where $B_\lambda$ is defined
as $B_\lambda=\{x:\ \|x\|\leq \lambda\},\lambda= \  \max\{\lambda_0,\lambda_1, \cdots \lambda_m\}.$ \\
Define a map $ F : B_\lambda \rightarrow X $ such that
\begin{align*}
\ Fx(t)& =x_0+\sum_{0<t_k<t}I_k(x(t_k^-))+\frac{1}{\Gamma(\alpha)}\int_0^t(t-s)^{\alpha-1}f_{1}(s,x(s))ds
\\ & + \frac{1}{\Gamma(\alpha)}\int_0^t(t-s)^{\displaystyle{\alpha-1}}f_{2}(s,x(s))ds.
\end{align*}
Let us consider
$$F_{1}x(t)=x_0+\sum_{0<t_k<t}I_k(x(t_k^-))+\frac{1}{\Gamma(\alpha)}\int_0^t(t-s)^{\alpha-1}f_{1}(s,x(s))ds$$ and
$$F_{2}x(t)=\frac{1}{\Gamma(\alpha)}\int_0^t(t-s)^{\alpha-1}f_{2}(s,x(s))ds.$$

\emph{Step:1} $F$ is self mapping. \\
Now we show that $F(B_r)\subset B_r.$ For $t\in [0,t_1]$
\begin{eqnarray}
\ |Fx(t)|
&& \le |x_0|+\frac{1}{\Gamma(\alpha)}\int_0^{t}(t-s)^{\alpha-1}\ |f(s.x(s))| ds \nonumber \\
&& \le |x_0|+\frac{1}{\Gamma(\alpha)}\int_0^{t}(t-s)^{\alpha-1}\ |f_{1}(s,x(s))|ds\nonumber \\
&&+\frac{1}{\Gamma(\alpha)} \int_0^{t}(t-s)^{\alpha-1}\ |f_{2}(s,x(s))|ds \nonumber \\
&& \le |x_0|+\frac{1}{\Gamma(\alpha)}\int_0^{t}(t-s)^{\alpha-1}M_{1}(s)ds \nonumber \\
&&+\frac{1}{\Gamma(\alpha)}\int_0^{t}(t-s)^{\alpha-1}M_{2}(s)ds \nonumber
\\
&& \le |x_0|+\frac{1}{\Gamma(\alpha)}\Big( \int_0^{t}(t-s)^{\frac{\alpha-1}{1-p}}ds\Big)^{1-p} \Big(\int_0^{t}M_1^{\frac{1}{p}}(s)ds\Big)^p
\nonumber \\
&&+\frac{1}{\Gamma(\alpha)}\Big( \int_0^{t}(t-s)^{\frac{\alpha-1}{1-p}}ds\Big)^{1-p} \Big(\int_0^{t}M_2^{\frac{1}{p}}(s)ds\Big)^p \nonumber \\
&& \le |x_0|+\frac{c \Big(\ \|M_{1}\|_{\frac{1}{p}}+\ \|M_{2}\|_{\frac{1}{p}}\Big)}{\Gamma(\alpha)}T^{\alpha-p} = \lambda_0. \nonumber
\end{eqnarray}
For $t\in (t_1,t_2]$
\begin{eqnarray}
\ |Fx(t)|
&& \le |x_0|+|I_1(x(t_1^-))|+\frac{1}{\Gamma(\alpha)}\int_0^{t}(t-s)^{\alpha-1}\ \|f(s.x(s))\| ds \nonumber \\
&& \le |x_0|+l_1+\frac{1}{\Gamma(\alpha)}\int_0^{t}(t-s)^{\alpha-1}\ |f_{1}(s,x(s))|ds\nonumber \\
&&+\frac{1}{\Gamma(\alpha)} \int_0^{t}(t-s)^{\alpha-1}\ |f_{2}(s,x(s))|ds \nonumber \\
&& \le |x_0|+l_1+\frac{1}{\Gamma(\alpha)}\int_0^{t}(t-s)^{\alpha-1}M_{1}(s)ds\nonumber \\
&&+ \frac{1}{\Gamma(\alpha)}\int_0^{t}(t-s)^{\alpha-1}M_{2}(s)ds \nonumber \\
&& \le |x_0|+l_1+\frac{1}{\Gamma(\alpha)}\Big( \int_0^{t}(t-s)^{\frac{\alpha-1}{1-p}}ds\Big)^{1-p} \Big(\int_0^{t}M_1^{\frac{1}{p}}(s)ds\Big)^p
 \nonumber
\\
&&+\frac{1}{\Gamma(\alpha)}\Big( \int_0^{t}(t-s)^{\frac{\alpha-1}{1-p}}ds\Big)^{1-p} \Big(\int_0^{t}M_2^{\frac{1}{p}}(s)ds\Big)^p \nonumber \\
&& \le |x_0|+l_1+\frac{c \Big(\ \|M_{1}\|_{\frac{1}{p}}+\ \|M_{2}\|_{\frac{1}{p}}\Big)}{\Gamma(\alpha)}T^{\alpha-p} = \lambda_1. \nonumber
\end{eqnarray}
For $t\in (t_2,t_3]$
\begin{eqnarray}
\ |Fx(t)|
&& \le |x_0|+\sum_{k=0}^2|I_k(x(t_k^-))|+\frac{1}{\Gamma(\alpha)}\int_0^{t}(t-s)^{\alpha-1}\ |f(s.x(s))| ds \nonumber \\
&& \le |x_0|+2l_1+\frac{1}{\Gamma(\alpha)}\int_0^{t}(t-s)^{\alpha-1}\ |f_{1}(s,x(s))|ds\nonumber \\
&&+\frac{1}{\Gamma(\alpha)} \int_0^{t}(t-s)^{\alpha-1}\ |f_{2}(s,x(s))|ds \nonumber
\\
&& \le |x_0|+2l_1+\frac{1}{\Gamma(\alpha)}\int_0^{t}(t-s)^{\alpha-1}M_{1}(s)ds\nonumber \\
&&+ \frac{1}{\Gamma(\alpha)}\int_0^{t}(t-s)^{\alpha-1}M_{2}(s)ds \nonumber
\\
&& \le |x_0|+2l_1+\frac{1}{\Gamma(\alpha)}\Big( \int_0^{t}(t-s)^{\frac{\alpha-1}{1-p}}ds\Big)^{1-p} \Big(\int_0^{t}M_1^{\frac{1}{p}}(s)ds\Big)^p \nonumber \\
&&+\frac{1}{\Gamma(\alpha)}\Big( \int_0^{t}(t-s)^{\frac{\alpha-1}{1-p}}ds\Big)^{1-p} \Big(\int_0^{t}M_2^{\frac{1}{p}}(s)ds\Big)^p \nonumber \\
&& \le |x_0|+2l_1+\frac{c \Big(\ \|M_{1}\|_{\frac{1}{p}}+\ \|M_{2}\|_{\frac{1}{p}}\Big)}{\Gamma(\alpha)}T^{\alpha-p} = \lambda_2. \nonumber
\end{eqnarray}
For $t\in (t_m,T]$
\begin{eqnarray}
\ |Fx(t)|
&& \le |x_0|+\sum_{k=0}^m|I_k(x(t_k^-))|+\frac{1}{\Gamma(\alpha)}\int_0^{t}(t-s)^{\alpha-1}\ |f(s.x(s))| ds \nonumber \\
&& \le |x_0|+ml_1+\frac{1}{\Gamma(\alpha)}\int_0^{t}(t-s)^{\alpha-1}\ |f_{1}(s,x(s))|ds\nonumber \\
&&+ \frac{1}{\Gamma(\alpha)} \int_0^{t}(t-s)^{\alpha-1}\ |f_{2}(s,x(s))|ds \nonumber
\end{eqnarray}
\begin{eqnarray}
&& \le |x_0|+ml_1+\frac{1}{\Gamma(\alpha)}\int_0^{t}(t-s)^{\alpha-1}M_{1}(s)ds\nonumber \\
&&+ \frac{1}{\Gamma(\alpha)}\int_0^{t}(t-s)^{\alpha-1}M_{2}(s)ds \nonumber \\
&& \le |x_0|+ml_1+\frac{1}{\Gamma(\alpha)}\Big( \int_0^{t}(t-s)^{\frac{\alpha-1}{1-p}}ds\Big)^{1-p} \Big(\int_0^{t}M_1^{\frac{1}{p}}(s)ds\Big)^p \nonumber \\
&&+\frac{1}{\Gamma(\alpha)}\Big( \int_0^{t}(t-s)^{\frac{\alpha-1}{1-p}}ds\Big)^{1-p} \Big(\int_0^{t}M_2^{\frac{1}{p}}(s)ds\Big)^p \nonumber \\
&& \le |x_0|+ml_1+\frac{c \Big(\ \|M_{1}\|_{\frac{1}{p}}+\ \|M_{2}\|_{\frac{1}{p}}\Big)}{\Gamma(\alpha)}T^{\alpha-p} = \lambda_m \nonumber
\end{eqnarray}

and thus $F(B_\lambda)\subset B_\lambda.$ \\

\emph{Step:2} $F_1$ is continuous and $\gamma-$ contraction.

To prove the continuity of $F_1$ for $t\in [0,T]$, let us consider a sequence $x_n$
converging to $x$. Taking the norm of $F_1x_n(t)-F_1x(t)$, we have
\begin{align*}
&\ |F_1x_n(t)-F_1x(t)|  \\
& \leq  \sum_{0<t_k<t}I_1(x_n(t_k^-)-x(t_k^-))+ \frac{1}{\Gamma \alpha}\int_0^t (t-s)^{\alpha-1}L(s) |x_n(s)-x(s)|ds \\
&\leq \sum_{0<t_k<t}l_2|(x_n(t_k^-)-x(t_k^-)| \\ &+ \frac{1}{\Gamma \alpha}\Big(\int_0^t (t-s)^{\frac{\alpha-1}{1-p}}ds \Big)^{1-p}\Big(\int_0^t L_1^{\frac{1}{p}}(s)ds \Big)^p \|x_n-x\| \\
& \leq (ml_2 + \frac{c \|L_1\|_{\frac{1}{p}} T^{\alpha-p}}{\Gamma{\alpha}})\|x_n-x\|.
\end{align*}
From the above analysis we obtain
$$\|F_1x_n-F_1x\| \le (ml_2+\frac{c \|L\|_{\frac{1}{p}}T^{\alpha-p}}{\Gamma (\alpha)})\|x_n-x\|.$$

To prove that $F_1$ is $\gamma_1-$ a contraction, let us consider
for $x, y\in B_r$,
\begin{align*}
&\ |F_1x(t)-F_1y(t)|  \\
& \leq  \sum_{0<t_k<t}I_k(x(t_k^-)-y(t_k^-))+ \frac{1}{\Gamma (\alpha)}\int_0^t(t-s)^{\displaystyle{\alpha-1}}L_1(s) |x(s)-y(s)|ds \\
&\leq \sum_{0<t_k<t}l_2|(x(t_k^-)-y(t_k^-)|\\ &+ \frac{1}{\Gamma (\alpha)}\Big(\int_0^t(t-s)^{\frac{\alpha-1}{1-p}}ds\Big)^{1-p}\Big(\int_0^t L_1^{\frac{1}{p}}(s)ds\Big)^p \|x-y\| \\
& \leq (ml_2 + \frac{c \|L_1\|_{\frac{1}{p}} T^{\alpha-p}}{ \Gamma(\alpha)})\|x-y\|.
\end{align*}
Thus for $$\gamma_1=(ml_2+\frac{c \|L_1\|_{\frac{1}{p}} T^{\alpha-p}}{ \Gamma (\alpha)})<1, $$ $F_1$ is $\gamma_1-$ contraction. \\

\emph{Step:3} $F_2$ is compact. For $0\le \tau_1 \le \tau_2 \le T$, we have
\begin{align*}
&\ |F_2x(\tau_2)-F_2x(\tau_1)|  \\
&\leq   \frac{1}{\Gamma (\alpha)}\Big |\int_0^{\tau_2}(\tau_2-s)^{\alpha-1}f_2(s,x(s))ds - \int_0^{\tau_1} (\tau_1-s)^{\alpha-1}f_2(s,x(s))ds\Big |  \\
&\leq  \frac{1}{\Gamma (\alpha)}\Big |\int_0^{\tau_1}(\tau_2-s)^{\alpha-1}f_2(s,x(s))ds  +\int_{\tau_1}^{\tau_2} (\tau_2-s)^{\alpha-1}f_2(s,x(s))ds\\
&\quad -\int_0^{\tau_1} (\tau_1-s)^{\alpha-1}f_2(s,x(s))ds\Big |\\
&\leq \frac{1}{\Gamma (\alpha)}\int_0^{\tau_1}\Big((\tau_1-s)^{\alpha-1}-(\tau_2-s)^{\alpha-1}\Big) |f_2(s,x(s))|ds \\
&\quad  +\frac{1}{\Gamma (\alpha)}\int_{\tau_1}^{\tau_2}(\tau_2-s)^{\alpha-1}|f_2(s,x(s))|ds  \quad (\because (\tau-s)^{\alpha-1}  \ \mbox{is a decreasing function of} \ \tau-s.)\\
&\leq \frac{1}{\Gamma (\alpha)}\int_0^{\tau_1}\Big((\tau_1-s)^{\alpha-1}-(\tau_2-s)^{\alpha-1}\Big)M_2(s)ds \\
&\quad  +\frac{1}{\Gamma (\alpha)}\int_{\tau_1}^{\tau_2}(\tau_2-s)^{\alpha-1}M_2(s)ds
\end{align*}
\begin{align*}
&\leq \frac{1}{\Gamma (\alpha)}\Big(\Big (\int_0^{\tau_1}\Big((\tau_1-s)^{\alpha-1}-(\tau_2-s)^{\alpha-1}\Big)^\frac{1}{1-p}ds \Big)^{1-p}\Big)  \\
& \times \Big(\int_0^{\tau_1}M_2^{\frac{1}{p}}(s)ds\Big)^p \\
&\quad  +\frac{1}{\Gamma (\alpha)}\Big (\int_{\tau_1}^{\tau_2}(\tau_2-s)^{\frac{\alpha-1}{1-p}}ds \Big)^{1-p}\Big(\int_0^{\tau_1}M_2^{\frac{1}{p}}(s)ds\Big)^p \\
&\leq \frac{1}{\Gamma (\alpha)}\Big(\Big (\int_0^{\tau_1}\Big((\tau_1-s)^{\frac{\alpha-1}{1-p}}-(\tau_2-s)^{\frac{\alpha-1}{1-p}}\Big)ds \Big)^{1-p}\Big)          \quad (\because x^z-y^z\geq (x-y)^z \quad \forall x\geq y\geq 0, z>1.) \\
& \times \Big(\int_0^{\tau_1}M_2^{\frac{1}{p}}(s)ds\Big)^p \\
&\quad  +\frac{1}{\Gamma (\alpha)}\Big (\int_{\tau_1}^{\tau_2}(\tau_2-s)^{\frac{\alpha-1}{1-p}}ds \Big)^{1-p}\Big(\int_0^{\tau_1}M_2^{\frac{1}{p}}(s)ds\Big)^p \\
&\leq \frac{c}{\Gamma (\alpha)}\Big(\tau_1^{\frac{\alpha-p}{1-p}}-\tau_2^{\frac{\alpha-p}{1-p}}+(\tau_2-\tau_1)^\frac{\alpha-p}{1-p}\Big)^{1-p}         \|M_2\|_{\frac{1}{p}}+\frac{c}{\Gamma (\alpha)}(\tau_2-\tau_1)^{\alpha-p}\|M_2\|_{\frac{1}{p}} \\
&\leq \frac{c}{\Gamma (\alpha)}\Big((\tau_2-\tau_1)^\frac{\alpha-p}{1-p}\Big)^{1-p}         \|M_2\|_{\frac{1}{p}}+\frac{c}{\Gamma (\alpha)}(\tau_2-\tau_1)^{\alpha-p}\|M_2\|_{\frac{1}{p}} \quad (\because \tau_2 >\tau_1.) \\
& \leq \frac{2c \|M_2\|_{\frac{1}{p}}}{\Gamma(\alpha)}(\tau_2-\tau_1)^{\alpha-p}. 
\end{align*}
The right-hand side of the above expression does not depend on
$x$. Thus using Arzel-Ascoli theorem for equi-continuous functions
(Diethelm, Theorem D.10 \cite{diet}), we conclude that $F_2(B_r)$
is relatively compact and hence $F_2$ is completely continuous on $ I- \{t_1,t_2 \cdots t_m \}.$ In similar way it can be prove the equi-continuity of $F$ on $ t=t_k^- \ \mbox{and} \ t=t_k^+ , k=1,2, \cdots, m.$ And thus $F_2$ is compact on $[0,T].$  \\ \\
\emph{Step:4} $F$  is condensing. \\
As $F = F_1+F_2, F_1 $ is continuous, contraction and $F_2$ is compact, so using the Lemma \ref{lem1*}, $F$ is condensing map on $B_r$.\\
 And hence using the Theorem \ref{thm1} we conclude that \eqref{EE_2}  has a solution in $B_r$.
\end{proof}

\begin{theorem}
If $f$ is bounded and Lipschitz, in particular,
$|f(t,x)-f(t,y)| \le L_1^*(t)|x-y| \ \mbox{for all} \ (t,x), (t,y) \in R \ \mbox{and} \ L_1^* \in L_{1/p}([0,T],\mathbb{R}^{+})$,then the problem \eqref{EE_2} has unique solution in $B_\lambda$, provided that
$$\gamma_1^*=ml_2+\frac{c \|L_1^*\|{\frac{1}{p}} T^{\alpha-p}}{\Gamma(\alpha+1)}<1, \quad \mbox{where} \  c=\Big(\frac{1-p}{\alpha-p}\Big)^{1-p}.$$
\end{theorem}

\section{Impulsive fractional differential equations with finite delay}
In this section, we discuss existence and uniqueness of
solutions of the following \emph{ impulsive fractional differential equations of Caputo type with finite delay} in a Banach space $X$ with norm $|.|$
\begin{align}
\begin{split}
^{C}D^{\alpha}x(t) & =
f(t,x_t), \; t \in [0,T], \; t\neq t_k,
\\ \Delta x(t)|_{t=t_k} & = I_k(x(t_k^-)), k =1,2,...,m , \; \\   x(t) &=\phi (t), t \in [-r,0].
\end{split}
\label{EE_4}
\end{align}
where $f:I \times \mathcal{C} \rightarrow X $ , $ \mathcal{C} = C([-r,0],X) $.$I_k \in C(X,X), (k=1,2,...,m)$ and $X$ is a Banach space with a norm $|.|$. For any $ x:[-r,T]\rightarrow X \ \mbox{and} \ t\in I, x_t\in \mathcal{C} $ and defined by $ x_t(s)=x(t+s), s\in [-r,0].$ Here, our tools are Banach, and Schaefer fixed point theorems. \\
Define a new Banach space $PC([-r,T],X)$
\begin{align*}
PC([-r,T],X) = \{ x:[-r,T] \rightarrow X| x\in C((t_k,t_{k+1}],X)\cup C([-r,0],X), k = 0,1,2 \cdots m,  \\   x(t_k^+) \ \mbox{and} \ x(t_k^-) \ \mbox{exist}, \ x(t_k)=x(t_k^-) \}.
\end{align*} with sup-norm $\|.\|$, defined by $\|x\|=\sup \{|x(t)|: t\in [-r,T]\}.$

\begin{definition}
A solution of fractional differential equation \eqref{EE_4} is a piecewise continuous function $x\in PC([-r,T],X)$ which satisfies \eqref{EE_4}.
\end{definition}

Consider the initial value problem \eqref{EE_4} on  $I \times \mathcal{C}$ for some fixed $T>0$, and assume that there exist $p \in (0,\alpha), \alpha \in (0,1), M_{3}, M_{4},  L_2  \in L_{1/p}([0,T],\mathbb{R}^{+})$ such that the following assumptions are satisfied
\begin{enumerate}
\item[(A.4)] $f \in C(I\times \mathcal{C},X)$
\item[(A.5)] $f$ bounded, in particular, $|f(t,\phi)| \le M_{3}(t)$ for all $(t,\phi) \in I \times \mathcal{C}$,

\item[(A.6)] $f$ is Lipschitz, in particular,  $ |f(t,\phi)-f(t,\psi)| \le L_2(t)\ \| \phi-\psi \|$ for all
$(t,\phi), (t,\psi) \in I \times \mathcal{C}.$
\end{enumerate}

\begin{lemma}(Fecken et al, Lemma 2. \cite{fecken}) The initial value problem \eqref{EE_4} is equivalent to the non-linear integral equation.
\begin{eqnarray}
\ x(t) &=& \phi(0)+ \frac{1}{\Gamma(\alpha)}\int_0^t(t-s)^{\alpha-1}f(s,x_s)ds  \; t\in [0,t_1] \nonumber
\\ &=& \phi(0)+I_1(x(t_1^-))+\frac{1}{\Gamma(\alpha)}\int_0^t(t-s)^{\alpha-1}f(s,x_s)ds , \; t\in (t_1,t_2] \nonumber
\\ &=& \phi(0)+\sum_{k=1}^2I_k(x(t_k^-))+\frac{1}{\Gamma(\alpha)}\int_0^t(t-s)^{\alpha-1}f(s,x_s)ds,  \; t\in (t_2,t_3] \nonumber
\\ &=& \phi(0)+\sum_{k=1}^mI_k(x(t_k^-))+\frac{1}{\Gamma(\alpha)}\int_0^t(t-s)^{\alpha-1}f(s,x_s)ds , \;  t\in (t_m,T]\nonumber
\\ &=& \phi(t) ,  \;  t\in [-r,0]  \label{EE_5}
\end{eqnarray}
In other words, every solution of the integral equation (\eqref{EE_5}) is
also solution of our original initial value problem \eqref{EE_4} and conversely.
\end{lemma}

\begin{remark}
Since history part/initial condition $x(t)=\phi(t), t\in [-r,0]$ is known, so we will investigate the existence and uniqueness of solution in $I=[0,T].$
\end{remark}

\begin{theorem}[Existence and Uniqueness of solution]
Under the assumptions $(A.3)-(A.6)$ the problem \eqref{EE_4} has
an unique solution in $[0,T]$, provided that
$$\gamma_2=(ml_2+\frac{c \|L_2\|_{\frac{1}{p}} T^{\alpha-p}}{ \Gamma (\alpha)})<1, \quad \mbox{where} \  c=\Big(\frac{1-p}{\alpha-p}\Big)^{1-p}.$$ \label{thm3}
\end{theorem}
\begin{proof}
 In this case we define the operator $F:PC(I,X) \rightarrow PC(I,X) $ by
  $$Fx(t) = \sum_{0<t_k<t}^mI_k(x_n(t_k^-)+ \frac{1}{\Gamma (\alpha)}\int_0^t (t-s)^{\displaystyle{\alpha-1}}f(s,x_s)ds.$$

 \emph{Step:1} To prove that $F$ is self mapping,  we need to prove that for each $x\in PC(I,X), Fx\in PC(I,X).$ \\
 We can see that the proof is similar to the proof of continuity of $F_1$ in Step:2 of the Theorem \ref{thm2} and hence we omit it.

\emph{Step:2} $F$ is continuous and $\gamma_2-$ contraction. \\
The proof of this step is also similar to the proof of continuous and $\gamma_1-$ contraction of $F_1$ in Step:2 of the Theorem \ref{thm2}. \\
Now by applying Banach's fixed point theorem, we get that the operator $F$ has an unique fixed point in $PC(I,X)$ and hence the problem \eqref{EE_4} has an unique solution in $PC([-r,T],X).$
\end{proof}

Our next result is based on Schaefer's fixed point theorem. In this case we replace assumption $A.3$ with the followings linear growth condition:
\begin{enumerate}
\item[(A.3')] $I_k$ bounded, in particular, $|I_k(x)| \le l_1^*$,
\item[(A.5')] $f$ bounded, in particular, $|f(t,\phi)| \le M_{4}(t)(1+\|\phi\|)$ for all $(t,\phi) \in I \times \mathcal{C}.$
\end{enumerate}

\begin{theorem}
Under the assumptions $A.3'$ and $A.5'$, problem \eqref{EE_4} has at least one solution. \label{thm4}
\end{theorem}

\begin{proof}
We transform the problem into a fixed point problem. For this purpose, consider the operator $F:PC(I,X) \rightarrow PC(I,X) $ defined by $$Fx(t) = \sum_{0<t_k<t}^mI_k(x_n(t_k^-)+ \frac{1}{\Gamma (\alpha)}\int_0^t
(t-s)^{\displaystyle{\alpha-1}}f(s,x_s)ds.$$

\emph{Step:1} $F$ is continuous. Let $\{x^n\}$ be a sequence such that $x^n\rightarrow x \ \mbox{in} \ PC(I,X).$ Then for every $t\in I$, we have
\begin{align*}
&\ |Fx^n(t)-Fx(t)|  \\
& \leq  \sum_{0<t_k<t}|I_k(x^n(t_k^-))-I_k(x(t_k^-))|+ \frac{1}{\Gamma \alpha}\int_0^t(t-s)^{\alpha-1}|f(s,x_s^n)-f(s,x_s)|ds \\
& \leq \sum_{0<t_k<t}|I_k(x^n(t_k^-))-I_k(x(t_k^-))|\\ &+\|f(.,x_.^n)-f(.,x_.)\| \frac{1}{\Gamma \alpha}\Big(\int_0^t(t-s)^{\alpha-1}\Big)^{1-p}ds \\
& \leq \sum_{k=1}^m\|I_k(x^n(.))-I_k(x(.))\| + \frac{T^{\alpha}}{\Gamma{(\alpha+1)}}\|f(.,x_.^n)-f(.,x_.)\|.
\end{align*}
We can see that if $n\rightarrow \infty, Fx^n \rightarrow Fx, \ \mbox{as}\ x^n\rightarrow x$ and $I_k \ \mbox{and}\ f$ both are continuous.  Hence $F$ is continuous.

\emph{Step:2} $F$ maps bounded sets into bounded sets in $PC(I,X).$ \\
It is enough to show that for any $\delta >0$, there exists a $ l>0$ such that $x\in B_{\delta} = \{x\in PC(I,X)| \|x\|\le \delta \},$ we have $\|Fx\|\leq l.$ \\
For $t\in [0,T],$ we have
\begin{eqnarray}
\ |Fx(t)|
&& \le |\phi (0)|+\sum_{0<t_k<t}|I_k(x(t_k^-))|+\frac{1}{\Gamma(\alpha)}\int_0^{t}(t-s)^{\alpha-1}\ |f(s,x_s)\| ds \nonumber \\
&& \le |x_0|+ml_1^*+\frac{1}{\Gamma(\alpha)}\int_0^{t}(t-s)^{\alpha-1}\ |f_{1}(s,x(s))|ds\nonumber
\\
&&+\frac{1}{\Gamma(\alpha)} \int_0^{t}(t-s)^{\alpha-1}\ |f_{2}(s,x(s))|ds \nonumber \\
&& \le |x_0|+ml_1^*+\frac{1}{\Gamma(\alpha)}\int_0^{t}(t-s)^{\alpha-1}M_{4}(s)ds\nonumber \\
&&+ \frac{1}{\Gamma(\alpha)}\int_0^{t}(t-s)^{\alpha-1}M_{4}(s)ds \nonumber
\end{eqnarray}
\begin{eqnarray}
&& \le |x_0|+ml_1^*+\frac{1}{\Gamma(\alpha)}\Big( \int_0^{t}(t-s)^{\frac{\alpha-1}{1-p}}ds \Big)^{1-p} \Big(\int_0^{t}M_4^{\frac{1}{p}}(s)ds\Big)^p \nonumber \\
&&+\frac{1}{\Gamma(\alpha)}\Big( \int_0^{t}(t-s)^{\frac{\alpha-1}{1-p}}ds \Big)^{1-p} \Big(\int_0^{t}M_4^{\frac{1}{p}}(s)ds\Big)^p \nonumber \\
&& \le |x_0|+ml_1^*+\frac{c \Big(\ \|M_{4}\|_{\frac{1}{p}}\Big)}{\Gamma(\alpha)}T^{\alpha-p} \nonumber
\end{eqnarray}

\emph{Step:3} $F$ maps bounded sets into equi-continuous sets in $PC(I,X).$ \\
The proof of this step is similar to the proof of compactness of $F_2$ in Step:3 of the Theorem \eqref{thm2}.\\
As a consequence of steps 1-3 together with PC-type Arzela-Ascoli theorem (Fecken et al, Theorem 2.11 \cite{fecken}) the map $F:PC(I,X) \rightarrow PC(I,X) $ is completely continuous.

\emph{Step:4} A priori bounds. Now we prove that the set \\
$E(F)=\{x\in PC(I,X)|x=\lambda Fx, \ \mbox{for some} \ \lambda \in (0,1)\}$ is bounded. \\
We observe that for  $ t\in [0,T] \ \mbox{and} \ x\in E(F), x(t)=\lambda Fx(t).$
\begin{align*}
\ |x(t)|
& \le |Fx(t)| \\
& \le |\phi (0)|+\sum_{0<t_k<t}|I_k(x(t_k^-))|+\frac{1}{\Gamma(\alpha)}\int_0^{t}(t-s)^{\alpha-1}\ |f(s,x_s)\| ds \\
& \le |\phi (0)|+ml_1^*+\frac{1}{\Gamma(\alpha)}\int_0^{t}(t-s)^{\alpha-1}\ M_4(s)(1+\|x_s\|)ds \\
& \le |\phi (0)|+ml_1^*+\frac{1}{\Gamma(\alpha)}\int_0^{t}(t-s)^{\alpha-1}\ M_4(s)(1+\|x\|)ds \\
& \le |\phi (0)|+ml_1^*+\frac{1+\|x\|}{\Gamma(\alpha)}\Big( \int_0^{t}(t-s)^{\frac{\alpha-1}{1-p}}ds \Big)^{1-p} \Big(\int_0^{t}M_4^{\frac{1}{p}}(s)ds\Big)^p
\\
& \le |\phi (0)|+ml_1^*+\frac{c(1+\|x\|) \Big(\ \|M_4\|_{\frac{1}{p}}\Big)}{\Gamma(\alpha)}T^{\alpha-p}
\end{align*}
and hence
 $$ \|x\| \leq \frac{|\phi (0)|+ml_1^*+\frac{c \Big(\ \|M_4\|_{\frac{1}{p}}\Big)}{\Gamma(\alpha)}T^{\alpha-p} }{1-\frac{c \Big(\ \|M_4\|_{\frac{1}{p}}\Big)}{\Gamma(\alpha)}T^{\alpha-p}}.$$  This shows that $E(F)$ is bounded.\\
As a consequence of Schaefer's fixed point theorem, the problem \eqref{EE_4} has at least one solution in $PC([-r,T],X).$
\end{proof}

\begin{theorem}
If $f$ is bounded and Lipschitz, in particular,  $ |f(t,\phi)-f(t,\psi)| \le L_2^*(t)\ \| \phi-\psi \|,
 \ \mbox{for all} \ (t,\phi), (t,\psi) \in I \times \mathcal{C} \ \mbox{and} \ L_2^* \in L_{1/p}([0,T],\mathbb{R}^{+})$, then problem \eqref{EE_4} has an unique solution in $PC([-r,T],X)$, provided that $$\gamma_2^*=ml_2+\frac{c \|L_2^*\|{\frac{1}{p}} T^{\alpha-p}}{\Gamma(\alpha+1)}<1, \quad \mbox{where} \  c=\Big(\frac{1-p}{\alpha-p}\Big)^{1-p}.$$
\end{theorem}

Further, we consider the following more general Caputo fractional differential equation
\begin{align}
\begin{split}
^{C}D^{\alpha}x(t) & =
f(t,x(t),x_t), \quad t \in [0,T], \quad t\neq t_k,
\\ & \Delta x(t)|_{t=t_k} = I_k(x(t_k^-)), \quad k =1,2,...,m , \\   & x(t) =\phi (t), \quad t \in [-r,0].
\end{split}
\label{EE_6}
\end{align}
where $f:I \times X \times \mathcal{C} \rightarrow X $ and $ \mathcal{C} = C([-r,0],X),  \ I_k \in C(X,X), \ (k=1,2,...,m)$ and $X$ is a separable real Banach space with the norm $|.|$. Here, our tools will be Banach and Schaefer fixed point theorems.

\begin{definition}
A solution of fractional differential equation \eqref{EE_6} is a piecewise continuous function $x\in PC([-r,T],X)$ which satisfies \eqref{EE_6}.
\end{definition}

Consider the initial value problem \eqref{EE_6} on $I \times \mathcal{C}\times X $ for some fixed $T>0$ and assume that there exist $p \in (0,\alpha)$, $M_{5}, M_{6}, L_3 ,L_4 \in L_{1/p}([0,T],\mathbb{R}^{+})$ such that the following assumptions are satisfied:
\begin{enumerate}
\item[(A.7)] $f \in C(I\times X \times \mathcal{C} ,X)$
\item[(A.8)] $f$ bounded, in particular, $|f(t,x,\phi)| \le M_{5}(t)$ for all $(t,x,\phi) \in I \times X \times \mathcal{C} $,

\item[(A.9)] $f$ is Lipschitz, in particular, $ |f(t,x,\phi)-f(t,y,\psi)| \le L_3(t)| x-y|+L_4(t)\|\phi-\psi\|$ for all
$(t,x,\phi), (t,y,\psi) \in I\times X \times \mathcal{C}.$
\item[(A.8')] $|f(t,x,\phi)|\leq M_6(t)(1+|x|+\|\phi\|).$
\end{enumerate}

\begin{lemma}(Fecken et al, Lemma 2. \cite{fecken}) The initial value problem \eqref{EE_6} is equivalent to the following non-linear integral equation
\begin{eqnarray}
\ x(t) &=& \phi(0)+ \frac{1}{\Gamma(\alpha)}\int_0^t(t-s)^{\alpha-1}f(s,x_s,x(s))ds  \; t\in [0,t_1] \nonumber
\\ &=& \phi(0)+I_1(x(t_1^-))+\frac{1}{\Gamma(\alpha)}\int_0^t(t-s)^{\alpha-1}f(s,x_s,x(s))ds , \; t\in (t_1,t_2] \nonumber
\\ &=& \phi(0)+\sum_{k=1}^2I_k(x(t_k^-))+\frac{1}{\Gamma(\alpha)}\int_0^t(t-s)^{\alpha-1}f(s,x_s,x(s))ds,  \; t\in (t_2,t_3] \nonumber
\\ &=& \phi(0)+\sum_{k=1}^mI_k(x(t_k^-))+\frac{1}{\Gamma(\alpha)}\int_0^t(t-s)^{\alpha-1}f(s,x_s,x(s))ds ,  \;  t\in (t_m,T] \nonumber
\\ &=& \phi(t) ,  \;  t\in [-r,0]  \label{EE_7}
\end{eqnarray}
In other words, every solution of the integral equation \eqref{EE_7} is also solution of our original initial value problem \eqref{EE_6} and
conversely.
\end{lemma}

\begin{theorem}
Under the assumptions $A.3, A.7, A.8, A.9,$ the problem \eqref{EE_6} has an unique solution, provided that $$\gamma_3=ml_2+\frac{c(\|L_3\|_{\frac{1}{p}}+\|L_4\|_{\frac{1}{p}})T^{\alpha-p}}{\Gamma (\alpha)}<1.$$
\end{theorem}

\begin{proof}
The proof is similar to the proof of Theorem \ref{thm3}.
\end{proof}
\begin{theorem}
Under the assumptions $A.3' , A.8',$ the problem \eqref{EE_6} has at least one solution in $PC([-r,T],X).$
\end{theorem}

\begin{proof}
The proof is similar to the proof of Theorem \ref{thm4}.
\end{proof}

\begin{theorem}
 If $f$ is bounded and Lipschitz, in particular,  $|f(t,x,\phi)-f(t,y,\psi)| \le L_3^*(t)| x-y|+L_4^*(t)\|\phi-\psi\|,  \ \mbox{for all} \ (t,x,\phi), (t,y,\psi) \in I\times X \times \mathcal{C} \ \mbox{and} \ L_3^* ,  L_4^* \in L_{1/p}([0,T],\mathbb{R}^{+})$, then the problem \eqref{EE_6} has an unique solution in $PC([-r,T],X)$, provided that $$\gamma_3^*=ml_2+\frac{c \Big(\|L_3^*\|_{\frac{1}{p}} + \|L_4^*\|{\frac{1}{p}}\Big)T^{\alpha-p}}{\Gamma(\alpha+1)}<1, \quad \mbox{where} \  c=\Big(\frac{1-p}{\alpha-p}\Big)^{1-p}.$$
\end{theorem}

\section{Examples}
\begin{example}
\textbf{Fractional impulsive  logistic equation:}\\
Consider the following class of fractional logistic equations
\begin{align}
\begin{split}
^{C}D^{\alpha}x(t) &=
x(t)\Big(a(t)-b(t)x(t))\Big), \; t \in [0,T], t\neq t_k,
\\ \Delta x(t)|_{t=t_k} & = I_k(x(t_k^-)), k=1,2,3,  ....,m \;
\\ x(0) & =x_0,
\end{split}
\label{EE_8}
\end{align}
where $a(t) \in [a_*, a^*]$ and $b(t) \in [b_*, b^*]$ with $a_*,
b_*>0.$
\end{example}
\begin{lemma}(Fecken et al, Lemma 2. \cite{fecken}) The initial value problem \eqref{EE_8} is equivalent to the non-linear integral equation
\begin{eqnarray}
\ x(t) &=& x_0+ \frac{1}{\Gamma(\alpha)}\int_0^t(t-s)^{\alpha-1}\Big(x(s)(a(s)-b(s)x(s))\Big)ds, \quad t\in [0,t_1] \nonumber \\
&=& x_0+I_1(x(t_1^-))+\frac{1}{\Gamma(\alpha)}\int_0^t(t-s)^{\alpha-1}x(s)a(s)ds \nonumber \\ &&-\frac{1}{\Gamma(\alpha)}\int_0^t(t-s)^{\alpha-1}b(s)x^2(s)ds, \quad t\in (t_1,t_2] \nonumber \\
&=& x_0+\sum_{k=1}^2I_k(x(t_k^-))+\frac{1}{\Gamma(\alpha)}\int_0^t(t-s)^{\alpha-1}x(s)a(s)ds \nonumber \\
&&-\frac{1}{\Gamma(\alpha)}\int_0^t(t-s)^{\alpha-1}b(s)x^2(s)ds),   \quad t\in (t_2,t_3] \nonumber \\
&=& x_0+ \sum_{k=1}^mI_k(x(t_k^-))+\frac{1}{\Gamma(\alpha)}\int_0^t(t-s)^{\alpha-1}x(s)a(s)ds \nonumber
\\
&&-\frac{1}{\Gamma(\alpha)}\int_0^t(t-s)^{\alpha-1}b(s)x^2(s)ds,  \quad t\in [t_m,T]. \label{EE_9}
\end{eqnarray}

In other words, every solution of the integral equation \eqref{EE_9} is also solution of our original initial value problem \eqref{EE_8} and conversely.
\end{lemma}
We can easily see that for the problem \eqref{EE_8} our functions are
$f_1(t,x)=a(t)x$ and  $f_2(t,x)=-b(t)x^2.$
It is not difficult to deduce that
$$|f_1(t,x)| \le a^* \|x\|+ml_1 \ \mbox{and}\ |f_1(t,x)-f_1(t,y)\ \le a^*\|x-y\|.$$ Also $|f_2(t,x)| \le b^* \|x\|^2.$
From the integral representation of problem (\ref{EE_8}) we get
 $$ \ |x(t)| \leq |x_0|+ ml_1+ \frac{1}{\Gamma(\alpha)}\int_0^t(t-s)^{\alpha-1}(a^*+b^*)\|x\|ds.$$ \\
Using Gronwall's inequality (Diethelm, Lemma 6.19 \cite{diet, ding}), we get  \\
\begin{align*}
 & \ |x(t)|\leq (|x_0|+ ml_1) \exp\Big(\frac{1}{\Gamma(\alpha)}\int_0^t(t-s)^{\alpha-1}(a^*+b^*)ds\Big) \\ \\
 & \leq (|x_0|+ ml_1) \exp\Big(\frac{a^*+ b^*}{\Gamma{(\alpha+1)}}\Big).
\end{align*}
Thus $x$ is bounded which implies that all the assumptions of
Theorem \ref{thm2} and hence there exists a solution of the
problem \eqref{EE_8}.

We give some more examples which are inspired by \cite{fecken}.
\begin{example}
Consider the following Caputo impulsive delay fractional differential equations
\begin{align}
\begin{split}
^{C}D^{\alpha}x(t) &=
\frac{e^{-\nu t}\|x_t\|}{(1+e^t)(1+\|x_t\|)}, \; t \in [0,1], \ t\neq t_1, \ \nu >0
\\ & \Delta x(t)|_{t=t_1} =\frac{1}{2}, \;
\\ & x(t) = \phi(t), \quad t\in [-r,0],  \\  & \phi(0) =0,
\end{split}
\label{EE_10}
\end{align}
\end{example}
Set $C_1=C([0,1],\mathbb{R^+}), f(t,\phi)= \frac{e^{-\nu t}\phi}{(1+e^t)(1+\phi)}, \; (t,\phi)\in [0,1]\times C_1.$ \\
Let $\phi_1, \phi_2 \in C_1$ and $t\in [0,1].$ Then, we have
\begin{align*}
\ |f(t,\phi_1)-f(t,\phi_2)| & = \frac{e^{-\nu t}}{1+e^t}|\frac{\phi_1}{1+\phi_1}-\frac{\phi_2}{1+\phi_2}|  = \frac{e^{-\nu t}|\phi_1-\phi_2|}{(1+e^t)(1+\phi_1)(1+\phi_2)} \\ & \leq  \frac{e^{-\nu t}|\phi_1-\phi_2|}{(1+e^t)}  \leq L^*(t)|\phi_1-\phi_2|, \ \mbox{where} \ L^*(t)=\frac{e^{-\nu t}}{2}.
\end{align*}
Again, for all $\phi \in C_1$ and each $t\in [0,T],$
\begin{align*}
\ |f(t,\phi)| & = \frac{e^{-\nu t}}{1+e^t}|\frac{\phi}{1+\phi}|  \leq  \frac{e^{-\nu t}}{(1+e^t)}  < m_1(t), \ \mbox{where} \ m_1(t)= \frac{e^{-\nu t}}{2}
\end{align*}
 For $t\in [0,1]$ and some $p\in (0,\alpha),L^*(t)= m_1(t) = \frac{e^{-\nu t}}{2}\in L_{\frac{1}{p}}([0,1],\mathbb{R^+})$  with $M_1^*=\|\frac{e^{-\nu t}}{10}\|_{\frac{1}{p}},$ we can arrive at the inequality  $\frac{1}{4}+\frac{cM_1^*}{\Gamma(\alpha)}<1.$ We can see that all the assumptions of the Theorem \ref{thm3} are satisfied and hence the problem \eqref{EE_10} has an unique solution in $[0,1]$.
\begin{example}
Consider the following Caputo impulsive delay fractional differential equation
\begin{align}
\begin{split}
^{C}D^{\alpha}x(t) &=
\frac{\|x_t\|}{(1+e^t)(1+\|x_t\|)}, \; t \in [0,T], \ t\neq t_1, \ \nu >0
\\ & \Delta x(t)|_{t=t_1} = \frac{1}{2},\;
\\ & x(t) = \phi(t), \quad t\in [-r,0],   \\ & \phi(0) =0,
\end{split}
\label{EE_11}
\end{align}
\end{example}
Set $C_2=C([0,1],\mathbb{R^+}), f(t,\phi)= \frac{\phi}{(1+e^t)(1+\phi)}, \; (t,\phi)\in [0,1]\times C_2.$ \\
Again, for all $\phi \in C_2$ and each $t\in [0,T],$
\begin{align*}
\ |f(t,\phi)| & = \frac{e^{-t}}{1+e^t}|\frac{\phi}{1+\phi}| \leq  \frac{e^{-t}}{(1+e^t)}  <m_2(t)(1+\|\phi\|), \ \mbox{where} \ m_2(t)= \frac{e^{-t}}{4}
\end{align*}
 For $t\in [0,1]$ and some $p\in (0,\alpha), m_2(t) =  \frac{e^{-t}}{4} \in L_{\frac{1}{p}}([0,1],\mathbb{R^+})$ with $M_2^*=\|\frac{e^{-t}}{4}\|_{\frac{1}{p}},$ we can arrive at the inequality  $\frac{1}{4}+\frac{cM}{\Gamma(\alpha)}<1.$ We can see that all the assumptions of the Theorem \ref{thm4} are satisfied and hence the problem \eqref{EE_11} has a solution in $[0,1]$.

\end{document}